\documentclass[11pt]{amsart}
\usepackage{amsmath,amssymb,amsthm}
\usepackage[margin=2.7cm]{geometry}
\usepackage{esint}
\usepackage[backend=biber,style=numeric,sorting=nyt,
  maxbibnames=99,giveninits=true,isbn=false,url=false,doi=false]{biblatex}
\newbibmacro{linktitle}[1]{%
  \iffieldundef{doi}
    {\iffieldundef{eprint}{#1}{\href{https://arxiv.org/abs/\thefield{eprint}}{#1}}}
    {\href{https://doi.org/\thefield{doi}}{#1}}}
\DeclareFieldFormat*{title}{\usebibmacro{linktitle}{\mkbibemph{#1}}}
\DeclareFieldFormat{journaltitle}{#1}
\DeclareFieldFormat[article]{volume}{\mkbibbold{#1}}
\DeclareFieldFormat{pages}{#1}
\renewbibmacro{in:}{\ifentrytype{article}{}{\printtext{\bibstring{in}\intitlepunct}}}
\appto{\bibsetup}{\emergencystretch=1em}
\usepackage{hyperref}
\newtheorem{theorem}{Theorem}[section]
\newtheorem{lemma}[theorem]{Lemma}
\newtheorem{proposition}[theorem]{Proposition}

\theoremstyle{definition}

\theoremstyle{remark}
\newtheorem{remark}[theorem]{Remark}

\newcommand{\R}{\mathbb{R}}
\newcommand{\Rnn}{\R^{3\times 3}}
\newcommand{\SO}{\mathrm{SO}(3)}
\newcommand{\cof}{\operatorname{cof}}
\newcommand{\dist}{\operatorname{dist}}

\newcommand{\Id}{\mathrm{Id}}
\newcommand{\avg}{\fint}

\numberwithin{equation}{section}
\title[A frame-indifferent rank-one convex non-quasiconvex function]{A frame-indifferent rank-one convex function\\ in three dimensions which is not quasiconvex}
\author{Sergio Conti, Georg Dolzmann, and Patrick Dondl}
\makeatletter
\def\@setauthors{%
  \begin{center}\normalfont
  {\scshape Sergio Conti$^{1}$, Georg Dolzmann$^{2}$, and Patrick Dondl$^{3}$}\\[3mm]
  {\footnotesize\itshape
  $^1$\,Institut f\"ur Angewandte Mathematik, Universit\"at Bonn, 53115 Bonn, Germany\\[1mm]
  $^2$\,Fakult\"at f\"ur Mathematik, Universit\"at Regensburg, 93040 Regensburg, Germany\\[1mm]
  $^3$\,Fakult\"at f\"ur Mathematik und Physik, Universit\"at Freiburg, 79104 Freiburg, Germany\\}
  \end{center}}
\makeatother

\date{October 22, 2026}

\begin{document}

\begin{abstract}
We construct a function $W:\Rnn\to\R$ which is frame-indifferent,
$W(QF)=W(F)$ for all $Q\in\SO$, rank-one convex, but not quasiconvex.
In a neighbourhood of $\SO$ the function is an explicit polynomial of degree four
in the strain $F^TF-\Id$; the failure of quasiconvexity occurs at $F=\Id$ through
a third-order mechanism of \v{S}ver\'ak type, with test functions whose gradients
are arbitrarily close to $\Id$. The polynomial is obtained from a geometrically linear example,
a rank-one convex, non-quasiconvex quartic polynomial of the symmetric part of the gradient,
by evaluating it at the strain $\frac12(F^TF-\Id)$, adding a uniaxial prestress and a cubic correction, and enlarging the quartic term. The local function is extended to all of $\Rnn$
by adding a frame-indifferent polyconvex function which vanishes near $\SO$. This work was carried out with the assistance of the AI model
Claude Opus 5.5; the proofs of the main results have been formally verified in Lean~4.
\end{abstract}

\maketitle

\section{Introduction}

A continuous function $f:\R^{m\times n}\to\R$ is \emph{quasiconvex} if
\begin{equation}\label{eq:qc}
f(F)\le \avg_{(0,1)^n} f(F+\nabla\psi)\,dx
\qquad\text{for all } F\in\R^{m\times n},\ \psi\in C^\infty_c((0,1)^n;\R^m),
\end{equation}
and \emph{rank-one convex} if $t\mapsto f(F+ta\otimes b)$ is convex for all
$F$, $a\in\R^m$, $b\in\R^n$. Quasiconvexity implies rank-one convexity; the
converse, asked by Morrey \cite{Morrey52} (see also \cite{Ball87}), fails for $m\ge3$, $n\ge2$ by
\v{S}ver\'ak's example \cite{Sverak92}. 
The case $m=2$ has attracted a lot of attention. Rank-one convex functions which are not quasiconvex have been recently constructed for
$n\ge3$ in \cite{Cassese} and for $n\ge4$ in \cite{ABGP}.
 In the
$2\times2$ case, rank-one convexity is known to imply quasiconvexity on diagonal matrices
\cite{Mueller99diag} and on upper triangular matrices \cite{HKL18}, partial results for
homogeneous isotropic integrands are given in \cite{GuerraKristensen22}, the Burkholder function
satisfies the quasiconvexity inequality with respect to Hessians \cite{GuerraBurkholder}, and for quartic
homogeneous polynomials which are invariant under left and right rotations the two notions
coincide \cite{BrunoPasqualotto}; see also \cite{PedregalSverak98,FaracoSzekelyhidi08,AIPS12,KirchheimKristensen16}
for further partial results and approaches to the $2\times2$ problem. 
A proof that
rank-one convexity implies quasiconvexity
in the $2\times2$ case has been proposed in \cite{Pedregal}.
The present paper
concerns the case $m=n=3$.

In nonlinear elasticity the stored energy density $W:\Rnn\to\R$ is
\emph{frame-indifferent},
\begin{equation}\label{eq:fi}
W(QF)=W(F)\qquad\text{for all } Q\in\SO,\ F\in\Rnn.
\end{equation}
Quasiconvexity is the natural condition for the
existence of minimizers \cite{Ball77,Ball02}, while rank-one convexity
(ellipticity) is the condition that can be checked in practice. It is natural
to ask whether 
the two notions are distinct also for frame-indifferent stored energy densities. \v{S}ver\'ak's construction does not respect \eqref{eq:fi}, and we
are not aware of a frame-indifferent example in the literature. Versions of Morrey's
question with the invariance requirements of elasticity are listed as open in
\cite{VMGN22}, mainly in the planar case, where it is also observed that
even for $n\ge3$ no isotropic example seems to be known.
We show that this is the case.

\begin{theorem}\label{thm:main}
There is a function $W\in C^{\infty}(\Rnn)$ which satisfies
\eqref{eq:fi}, is rank-one convex, and is not quasiconvex.
\end{theorem}

The function $W$ we construct
coincides in a neighbourhood of $\SO$ with the polynomial $W_0(F)=g(\frac12(F^TF-\Id))$,
$g$ given in \eqref{eq:g} below, and is not quasiconvex at $F=\Id$.

The function is not isotropic (i.e.\ not invariant under $F\mapsto FQ$); the
isotropic case remains open (see \cite[Conjecture~1]{Cassese} for a homogeneous
version, and \cite{GuerraKristensen22} for positive results on homogeneous isotropic
integrands). Recent work has produced several new examples of rank-one convex,
non-quasiconvex integrands, see \cite{Grabovsky18,Cassese,ABGP} and the
references therein; none of them is frame-indifferent. (The examples on $\R^{d\times d}$
in \cite[Theorem~1.4]{Cassese} are invariant under conjugation $F\mapsto QFQ^T$,
$Q\in O(d)$, and under transposition, but not under $F\mapsto QF$.)
In Section~\ref{sec:linear} we first give an elementary example in the geometrically
linear setting, i.e.\ a rank-one convex, non-quasiconvex polynomial which depends only on
the symmetric part of its argument (Theorem~\ref{thm:linear}); the classical
counterexamples depend on the skew-symmetric part and do not apply in this setting
\cite{BKS19}. The nonlinear function $W_0$ is
built from it: $g$ is the linear density plus the prestress $2e_{33}$, a cubic correction and a
larger quartic term, see Remark~\ref{rem:fromlinear}.

\smallskip
\noindent\emph{Strategy.} The proof has three steps.
\begin{enumerate}
\item (Local construction, Section~\ref{sec:local}.) We construct an explicit frame-indifferent
polynomial $W_0$ whose second variation 
on rank-one
matrices at $\Id$,
$Q(a,b)=\frac12D^2W_0(\Id)[a\otimes b,a\otimes b]$, is nonnegative and vanishes exactly on the three \v{S}ver\'ak modes
$R_1=e_1\otimes e_1$, $R_2=e_2\otimes e_2$, $R_3=e_3\otimes(e_1+e_2)$, and whose
cubic term on the span of these modes is $rsu$.
\item (Local rank-one convexity, Section~\ref{sec:rc}.) By frame indifference,
$\det(F^TF)\,D^2W_0(F)[a\otimes b,a\otimes b]$ is a polynomial $P(e,c,b)$ in the invariants
$e=\frac12(F^TF-\Id)$, $c=F^Ta$ and $b$. 
With an appropriate parametrization over $\R^{10}$, 
near each normalized mode $P$ has a nondegenerate minimum 
as one can check by 
 the positive definiteness of explicit $10\times10$ rational matrices
(Lemma~\ref{lemma:hess}). An elementary compactness argument shows that
$W_0$ is rank-one convex in a neighbourhood of $\SO$.
\item (Globalization, Section~\ref{sec:global}.) We add $K\Phi$, where
$\Phi=\chi\bigl(|F-\cof F|^2+(\det F-1)^2+\sigma|F|^2\bigr)$ with
$\chi$ smooth, convex, $\chi=0$ on $(-\infty,4\sigma]$ and $\chi(s)=(s-4\sigma)^4+\mathrm{const}$ for $s\ge5\sigma$. $\Phi$ is polyconvex, frame-indifferent, vanishes in a
neighbourhood of $\SO$, and is uniformly convex along rank-one lines away from a
(slightly larger) neighbourhood of $\SO$, with a modulus that dominates the
growth of $D^2W_0$.
\end{enumerate}

\noindent\emph{Formal verification and use of AI.} The proofs of Theorem~\ref{thm:main}
and Theorem~\ref{thm:linear} have been formalized in the proof assistant Lean~4 \cite{Lean4}, version \texttt{4.34.1},
using the library Mathlib \cite{Mathlib}
in the release \texttt{v4.34.1}, see
 \cite{LeanCode}. The formalization was written by the AI model Claude Opus 5.5
(Anthropic)
under the guidance of the authors. The same tool
was also used in the preparation of the paper; see the statement at the
end of the paper.

\noindent\emph{Notation.} $|\cdot|$ is the Euclidean (Frobenius) norm,
$(\SO)_\delta:=\{F\in\Rnn:\dist(F,\SO)<\delta\}$, $e_1,e_2,e_3$ is the standard basis,
$\avg$ denotes the average.
We write $\operatorname{sym}F:=\frac12(F+F^T)$, $\R^{3\times3}_{\mathrm{sym}}$ for the
symmetric matrices, and $a\odot b:=\frac12(a\otimes b+b\otimes a)$ for $a,b\in\R^3$.
Then $\operatorname{sym}(a\otimes b)=a\odot b$ and
$|a\odot b|^2=\frac12\bigl(|a|^2|b|^2+(a\cdot b)^2\bigr)\ge\frac12|a|^2|b|^2$.

\noindent\emph{Cofactors and minors.} For $F\in\Rnn$, $\cof F$ is the matrix of
cofactors. We use the standard identities
\begin{equation}\label{eq:cof}
F^T\cof F=(\cof F)\,F^T=(\det F)\Id,\qquad\cof(F^T)=(\cof F)^T,\qquad
\cof(AB)=\cof A\,\cof B .
\end{equation}
In particular, $\cof F$ is symmetric if $F$ is, $\cof Q=Q$ for $Q\in\SO$, and
\begin{equation}\label{eq:cofQ}
\cof(QF)=Q\cof F,\qquad\det(QF)=\det F\qquad\text{for all } Q\in\SO,\ F\in\Rnn .
\end{equation}
Let $T(F):=(F,\cof F,\det F)\in\R^9\times\R^9\times\R=\R^{19}$ be the vector of minors of $F$.
The minors are affine along rank-one lines: for $F\in\Rnn$ and $a,b\in\R^3$ there are
$V_{\cof}\in\R^9$, $V_{\det}\in\R$ such that
\begin{equation}\label{eq:Taffine}
T(F+ta\otimes b)=T(F)+t\,(a\otimes b,V_{\cof},V_{\det})\qquad\text{for all }t\in\R .
\end{equation}
Indeed, every minor of $F+ta\otimes b$ is the determinant of a square matrix of the form
$B+t\,c\otimes d$; expanding it by multilinearity in the columns, all terms containing two
or more columns of $t\,c\otimes d$ vanish, since these columns are parallel to $c$.
A function $f:\Rnn\to\R$ which is a convex function of $T(F)$ is called \emph{polyconvex}
\cite{Ball77}; by \eqref{eq:Taffine} it is rank-one convex.

We shall use the following standard fact; see e.g.\ \cite[Prop.~4.19]{Mueller99}, \cite[Prop.~5.13]{Dacorogna}.

\begin{lemma}[Periodic test functions]\label{lemma:periodic}
Let $f:\Rnn\to\R$ be continuous and quasiconvex, $F\in\Rnn$, and let
$\psi\in C^\infty(\R^3;\R^3)$ be $Y$-periodic, $Y=(0,2\pi)^3$. Then
$f(F)\le\avg_Y f(F+\nabla\psi)\,dx$.
\end{lemma}

\section{The geometrically linear case}\label{sec:linear}

In linearized elasticity, frame indifference is replaced by invariance under
infinitesimal rotations, $f(F+A)=f(F)$ for all skew-symmetric $A$, i.e.\
$f(F)=\tilde g(\operatorname{sym}F)$. In this
setting the mechanism of \v{S}ver\'ak \cite{Sverak92} gives an elementary example,
which is a polynomial of degree four and requires neither a local analysis near the
modes nor a globalization step. It is also the starting point of the nonlinear
construction in Sections~\ref{sec:local}--\ref{sec:global}, see Remark~\ref{rem:fromlinear}.
We are not aware of a previous example of this type: as observed in \cite{BKS19}, the
classical counterexamples, including \v{S}ver\'ak's, depend in an essential way on the
skew-symmetric part of the gradient. A related but different statement is
\cite[Theorem~1.5]{ABGP}, which gives a quartic polynomial on
$\R^{3\times3}_{\mathrm{sym}}$ that is convex along the directions $a\otimes a$ but not
quasiconvex with respect to Hessians $D^2\Psi$ of scalar functions. Here instead the test
fields are vector valued and convexity is required along all directions
$\operatorname{sym}(a\otimes b)$; the test field $(\sin x_1,\sin x_2,-\sin(x_1+x_2))$ used below is not a gradient.
For Hessians of functions of two variables, the corresponding cubic $xyz$ on
$\operatorname{span}\{e_1\otimes e_1,e_2\otimes e_2,(e_1+e_2)\otimes(e_1+e_2)\}$ was studied by
Pedregal and \v{S}ver\'ak \cite{PedregalSverak98}, who showed that its quasiconvex and rank-one
convex envelopes coincide at $0$, so that the mechanism gives no counterexample there. In our
construction the first two modes are of this form, but the third one,
$M_3=e_3\odot(e_1+e_2)$ in \eqref{eq:Mk}, is the symmetrized gradient of the plane wave
$e_3\sin(x_1+x_2)$ and is not of the form $k\otimes k$.
While finalizing this paper, we were informed that a result similar to Theorem~\ref{thm:linear} has been obtained independently~\cite{Gammer26}.

\begin{theorem}\label{thm:linear}
There is a function $f\in C^{\infty}(\Rnn)$ which satisfies
$f(F+A)=f(F)$ for all $F\in\Rnn$
and all skew-symmetric $A\in\Rnn$, is rank-one convex, and is not quasiconvex.
\end{theorem}
\begin{proof}
For $e=(e_{ij})\in\R^{3\times3}_{\mathrm{sym}}$ let %$|e|^2=\sum_{i,j}e_{ij}^2$,
\begin{equation}\label{eq:q}
q(e):=4e_{12}^2-2e_{11}e_{22}+4(e_{13}-e_{23})^2,
\end{equation}
\begin{equation}\label{eq:glin}
g_{\mathrm{lin}}(e):=q(e)+e_{11}e_{22}(e_{13}+e_{23})+\tfrac34|e|^4.
\end{equation}
We define $f$ as the polynomial
\begin{equation}\label{eq:flin}
f(F):=g_{\mathrm{lin}}(\operatorname{sym}F),\qquad F\in\Rnn .
\end{equation}
% The function $f$ is the polynomial defined in \eqref{eq:flin}.
The invariance is clear, since $\operatorname{sym}(F+A)=\operatorname{sym}F$. For
$F\in\Rnn$ and $a,b\in\R^3$, the second derivative of $t\mapsto f(F+ta\otimes b)$ at $t$ is nonnegative by
Lemma~\ref{lemma:linrc} below; hence $f$ is rank-one convex. Lemma~\ref{lemma:linnotqc}
below shows that $f$ is not quasiconvex at $0$.
\end{proof}

\begin{remark}\label{rem:linmodes}
The quadratic form $q$ in~\eqref{eq:q} is indefinite on $\R^{3\times3}_{\mathrm{sym}}$, e.g.\
$q(\operatorname{diag}(1,1,0))=-2$, but nonnegative on the cone $\{a\odot b\}$ of
symmetrized rank-one matrices, see \eqref{eq:qB} below. It vanishes on the symmetrized
\v{S}ver\'ak modes
\begin{equation}\label{eq:Mk}
M_1:=e_1\otimes e_1,\qquad M_2:=e_2\otimes e_2,\qquad M_3:=e_3\odot(e_1+e_2),
\end{equation}
but also on other symmetrized rank-one matrices, e.g.\ on
$\operatorname{diag}(1,0,-1)=(e_1+e_3)\odot(e_1-e_3)$. This does not matter: the proof of
rank-one convexity only uses that the products of the three linear expressions $e_{11}$, $e_{22}$,
$e_{13}+e_{23}$, which appear in the second variation of the cubic term, are controlled by
$\sqrt q$ on $\{a\odot b\}$, see \eqref{eq:linprod} below. In contrast to the nonlinear
setting of Section~\ref{sec:local}, the plane $L_{\mathrm{sym}}:=\operatorname{span}\{M_1,M_2,M_3\}$
contains many more symmetrized rank-one directions: for
$e=rM_1+sM_2+wM_3$ one has $\det e=-\frac14w^2(r+s)$, and all of $\{w=0,\ rs\le0\}$ and
$\{r+s=0\}$ consists of matrices of the form $a\odot b$. This is why $q$ cannot be
taken as $-rs$ on $L_{\mathrm{sym}}$ plus a penalization of the orthogonal complement,
as in \cite{Sverak92}; on $L_{\mathrm{sym}}$ one has $q=-2rs$, which is nonnegative on
these additional directions.
\end{remark}

\begin{remark}[Minimizers]\label{lemma:linmin}
 One can check that the set of minimizers of $f$ consists of the matrices whose symmetric part is either $e_+$ or $e_-$, where $e_\pm:=\pm\sqrt{w_*(w_*-2)}(M_1+M_2)+w_*M_3$, with
 $w_*:=\frac{11-\sqrt{193}}{18}$. To see this, one first projects to 
 $L_{\mathrm{sym}}$, then restricts to $r=s$, and finally sets the two partial derivatives equal to zero.
\end{remark}

\begin{lemma}[Failure of quasiconvexity]\label{lemma:linnotqc}
Let $f$ be as in \eqref{eq:flin}, $\varepsilon>0$, and let $u_\varepsilon:\R^3\to\R^3$,
\[
u_\varepsilon(x):=\varepsilon\bigl(\sin x_1,\ \sin x_2,\ -\sin(x_1+x_2)\bigr).
\]
Then
\[
\avg_Y f(\nabla u_\varepsilon)\,dx=-\frac{\varepsilon^3}{4}+\frac{63}{32}\varepsilon^4 .
\]
In particular, $f$ is not quasiconvex at $F=0$.
\end{lemma}
\begin{proof}
With $M_k$ as in \eqref{eq:Mk}, one has
$\operatorname{sym}\nabla u_\varepsilon=rM_1+sM_2+wM_3$ with $r=\varepsilon\cos x_1$,
$s=\varepsilon\cos x_2$, $w=-\varepsilon\cos(x_1+x_2)$, i.e.\ $e_{11}=r$, $e_{22}=s$,
$e_{13}=e_{23}=\frac w2$ and all other entries vanish. Hence
\[
f(\nabla u_\varepsilon)=-2rs+rsw+\tfrac34(r^2+s^2+w^2)^2 .
\]
Using $\avg_Y\cos x_1\cos x_2\, dx=0$, $\avg_Y\cos x_1\cos x_2\cos(x_1+x_2)\, dx=\frac14$ and
$\avg_Y(\cos^2x_1+\cos^2x_2+\cos^2(x_1+x_2))^2\, dx=\frac{21}8$, we obtain the formula.
For $0<\varepsilon<\frac{8}{63}$ the right-hand side is negative, so that
$\avg_Y f(\nabla u_\varepsilon)\,dx<0=f(0)$. Since $u_\varepsilon$ is $Y$-periodic,
Lemma~\ref{lemma:periodic} shows that $f$ is not quasiconvex.
\end{proof}

\begin{lemma}[Rank-one convexity]\label{lemma:linrc}
Let $f$ be as in \eqref{eq:flin}, $F\in\Rnn$, and $a,b\in\R^3$. Then
\[
D^2f(F)[a\otimes b,a\otimes b]=\frac{d^2}{dt^2}\Big|_{t=0}f(F+ta\otimes b)\ge0 .
\]
\end{lemma}
\begin{proof}
Let $e:=\operatorname{sym}F$, $M:=a\odot b$,
\[
\beta_1:=a_1b_2,\qquad \beta_2:=a_2b_1,\qquad \beta_3:=(a_1-a_2)b_3+a_3(b_1-b_2),\]
and, for symmetric $X$, $r(X):=X_{11}$, $s(X):=X_{22}$, $v(X):=X_{13}+X_{23}$, so that
the cubic term of $g_{\mathrm{lin}}$ is $rsv$. Since $\operatorname{sym}(F+ta\otimes b)=e+tM$,
\begin{equation}\label{eq:linD2}
D^2f(F)[a\otimes b,a\otimes b]=D^2g_{\mathrm{lin}}(e)[M,M]
=2q(M)+D^2(rsv)(e)[M,M]+\tfrac34D^2(|\cdot|^4)(e)[M,M],
\end{equation}
and we estimate the three terms separately.

\emph{Step 1: the quadratic term.} We have
\begin{equation}\label{eq:qB}
q(a\odot b)=(a_1b_2)^2+(a_2b_1)^2+\bigl((a_1-a_2)b_3+a_3(b_1-b_2)\bigr)^2 =|\beta|^2.
\end{equation}
Indeed, $4M_{12}^2-2M_{11}M_{22}=(a_1b_2+a_2b_1)^2-2(a_1b_2)(a_2b_1)=\beta_1^2+\beta_2^2$ and
$2(M_{13}-M_{23})=(a_1-a_2)b_3+a_3(b_1-b_2)=\beta_3$.

\emph{Step 2: products of the linear functionals.} We show
\begin{equation}\label{eq:linprod}
\bigl(r(M)s(M)\bigr)^2+\bigl(r(M)v(M)\bigr)^2+\bigl(s(M)v(M)\bigr)^2\le3|a|^2|b|^2|\beta|^2 .
\end{equation}
First,
\begin{equation}\label{eq:linid}
r(M)s(M)=\beta_1\beta_2,\qquad
r(M)v(M)=\tfrac12a_1b_1\,\beta_3+a_1b_3\,\beta_2+a_3b_1\,\beta_1 .
\end{equation}
The first identity is $a_1b_1a_2b_2=(a_1b_2)(a_2b_1)$. For the second,
$2v(M)=(a_1+a_2)b_3+a_3(b_1+b_2)=\beta_3+2a_2b_3+2a_3b_2$, and $a_1b_1a_2b_3=a_1b_3\beta_2$,
$a_1b_1a_3b_2=a_3b_1\beta_1$. By \eqref{eq:linid} and the Cauchy--Schwarz inequality,
$(r(M)s(M))^2\le|a|^2|b|^2\beta_1^2\le|a|^2|b|^2|\beta|^2$ and
\[
\bigl(r(M)v(M)\bigr)^2\le\bigl(\tfrac14a_1^2b_1^2+a_1^2b_3^2+a_3^2b_1^2\bigr)(\beta_1^2+\beta_2^2+\beta_3^2)\le|a|^2|b|^2|\beta|^2.
\]
Exchanging the first and second components of $a$ and of $b$ replaces $\beta$ by
$(\beta_2,\beta_1,-\beta_3)$, leaves $|\beta|$ and $v(M)$ invariant, and exchanges $r(M)$ and
$s(M)$. Hence we also obtain
$(s(M)v(M))^2\le|a|^2|b|^2|\beta|^2$, and \eqref{eq:linprod} follows.

\emph{Step 3: the cubic term.} We have
\begin{align*}
D^2(rsv)(e)[M,M]&=2\bigl(r(e)\,s(M)v(M)+s(e)\,r(M)v(M)+v(e)\,r(M)s(M)\bigr)\\
&=2\langle e,\ s(M)v(M)M_1+r(M)v(M)M_2+r(M)s(M)M_3\rangle,
\end{align*}
with $M_1,M_2,M_3$ as in \eqref{eq:Mk}. Since $M_1,M_2,M_3$ are orthonormal, \eqref{eq:linprod} and $|a|^2|b|^2\le2|M|^2$ give
\[
D^2(rsv)(e)[M,M]\ge-2\sqrt{3}\,|a||b|\,|\beta|\,|e|\ge-2\sqrt6\,|M|\,|\beta|\,|e| .
\]

\emph{Step 4: conclusion.} Finally,
$D^2(|\cdot|^4)(e)[M,M]=4|e|^2|M|^2+8\langle e,M\rangle^2\ge4|e|^2|M|^2$. Inserting this,
\eqref{eq:qB} and Step~3 into \eqref{eq:linD2} we obtain
\[
D^2f(F)[a\otimes b,a\otimes b]\ge 2|\beta|^2-2\sqrt6\,|M|\,|\beta|\,|e|+3|M|^2|e|^2
=\bigl(\sqrt2|\beta|-\sqrt3\,|M||e|\bigr)^2\ge0 . \qedhere
\]
\end{proof}

\begin{remark}
The two mechanisms which simplify the linear case are the following. First, the cubic
term satisfies $D^3g_{\mathrm{lin}}(0)[X,M_k,M_k]=0$ for all $X$ and $k$ automatically, since two of
the three linear functionals $e_{11}$, $e_{22}$, $e_{13}+e_{23}$ vanish on each $M_k$; in the nonlinear case the
corresponding condition is a nontrivial constraint, which is satisfied by adding a cubic
correction, see Remark~\ref{rem:fromlinear}. Second, the convex quartic term $\frac34|e|^4$ controls the cubic
term both near the modes and at infinity, since $q$ vanishes on $\{a\odot b\}$ only to
second order at the modes and the cubic term is at most linear in $e$ in the second
variation. The degree four is the smallest possible for a polynomial example: if a polynomial $f$ of
degree three is rank-one convex, then its cubic part $f_c$ satisfies $D^3f_c[X,M,M]=0$ for all $X$
and all rank-one $M$ (since $D^2f(F)[M,M]$ is affine in $F$ and nonnegative); thus $f_c$ is
rank-one affine, hence a null Lagrangian, and $f$ is quasiconvex, because its quadratic part is
nonnegative on rank-one matrices and therefore quasiconvex (by Plancherel's theorem).
\end{remark}

\section{The local function}\label{sec:local}

With the quadratic form $q$ of \eqref{eq:q}, for $e\in\R^{3\times3}_{\mathrm{sym}}$ let
\begin{equation}\label{eq:g}
g(e):=2e_{33}+q(e)+e_{11}e_{22}\,(e_{11}+e_{22}+e_{13}+e_{23})
+(e_{13}+e_{23})^2(e_{11}+e_{22}-4e_{12})+3|e|^4 .
\end{equation}
We set
\begin{equation}\label{eq:w}
W_0(F):=g\bigl(\tfrac12(F^TF-\Id)\bigr),\qquad F\in\Rnn .
\end{equation}
$W_0$ is a polynomial of degree $8$, $W_0(\Id)=0$, and $W_0(QF)=W_0(F)$ for all
$Q\in\SO$ (indeed for all $Q\in\mathrm{O}(3)$). The linear term $2e_{33}$ of $g$ is the
prestress $|Fe_3|^2-1$, a convex function of $F$.

\begin{lemma}[Second variation at the identity]\label{lemma:Q}
Let $Q(a,b):=\frac12D^2W_0(\Id)[a\otimes b,a\otimes b]$. Then
\[
Q(a,b)=(a_1b_2)^2+(a_2b_1)^2+\bigl((a_1-a_2)b_3+a_3(b_1-b_2)\bigr)^2+|a|^2b_3^2 .
\]
In particular $Q\ge0$, and for $a,b\ne0$ one has $Q(a,b)=0$ if and only if
$a\otimes b$ is a multiple of one of
\[
R_1=e_1\otimes e_1,\qquad R_2=e_2\otimes e_2,\qquad R_3=e_3\otimes(e_1+e_2).
\]
\end{lemma}
\begin{proof}
With $A=a\otimes b$ one has
$\frac12\bigl((\Id+tA)^T(\Id+tA)-\Id\bigr)=t\,a\odot b+\frac{t^2}2|a|^2\,b\otimes b$, and $Q(a,b)$
is the coefficient of $t^2$ in $g$ evaluated at this matrix, which only involves the linear
and the quadratic part of $g$: it equals $|a|^2b_3^2+q(a\odot b)$, and \eqref{eq:qB}
gives the formula.

Let $a,b\ne0$ with $Q(a,b)=0$. Then $b_3=0$, and therefore $a_1b_2=a_2b_1=a_3(b_1-b_2)=0$.
If $a_3\ne0$, then $b_1=b_2=:\beta\ne0$, hence $a_1=a_2=0$ and
$a\otimes b=a_3\beta R_3$. If $a_3=0$ and $a_1\ne0$, then $b_2=0$, so $b_1\ne0$,
$a_2=0$, and $a\otimes b=a_1b_1R_1$. If $a_3=a_1=0$, then $a_2\ne0$, $b_1=0$, and
$a\otimes b=a_2b_2R_2$. Conversely $Q$ vanishes on the $R_i$.
\end{proof}

\begin{lemma}[Expansion on \v{S}ver\'ak's plane]\label{lemma:L}
Let $L:=\operatorname{span}\{R_1,R_2,R_3\}=\left\{\begin{pmatrix} r&0&0\\0&s&0\\u&u&0\end{pmatrix}\right\}$.
For $H\in L$,
\[
W_0(\Id+\varepsilon H)=-2\varepsilon^2rs+\varepsilon^3rsu+\varepsilon^4p(\varepsilon,r,s,u)
\]
for a polynomial $p$.
\end{lemma}
\begin{proof}
For $H\in L$,
$e=\frac12\bigl((\Id+\varepsilon H)^T(\Id+\varepsilon H)-\Id\bigr)$ has entries
$e_{11}=\varepsilon r+\frac{\varepsilon^2}2(r^2+u^2)$,
$e_{22}=\varepsilon s+\frac{\varepsilon^2}2(s^2+u^2)$,
$e_{12}=\frac{\varepsilon^2}2u^2$, $e_{13}=e_{23}=\frac{\varepsilon}2u$, $e_{33}=0$.
Hence, up to terms of order $\varepsilon^4$, $q(e)=-2e_{11}e_{22}+4e_{12}^2$ equals
$-2\varepsilon^2rs-\varepsilon^3\bigl(r(s^2+u^2)+s(r^2+u^2)\bigr)$, the two cubic terms of $g$
equal $\varepsilon^3rs(r+s+u)$ and $\varepsilon^3u^2(r+s)$, and $|e|^4=O(\varepsilon^4)$.
Summing up gives the claim.
\end{proof}

\subsection*{The reduced second variation}
For symmetric $e,X$ let $L_e(X):=Dg(e)[X]$ and $Q_e(X):=\frac12D^2g(e)[X,X]$, so that
$g(e+X)=g(e)+L_e(X)+Q_e(X)+O(|X|^3)$. For $G,A\in\Rnn$,
\[
\tfrac12\bigl((G+tA)^T(G+tA)-\Id\bigr)=e+tS+t^2K,\qquad e:=\tfrac12(G^TG-\Id),\quad
S:=\operatorname{sym}(G^TA),\quad K:=\tfrac12A^TA,
\]
so that
\begin{equation}\label{eq:D2W0}
D^2W_0(G)[A,A]=\frac{d^2}{dt^2}\Big|_{t=0}W_0(G+tA)=2Q_e(S)+2L_e(K).
\end{equation}
For $A=a\otimes b$ and $c:=G^Ta$ one has $S=c\odot b$ and
$K=\frac12|a|^2\,b\otimes b$, and, if $G$ is invertible,
$|a|^2=c^T(G^TG)^{-1}c=c^T\cof(\Id+2e)\,c/\det(\Id+2e)$ by \eqref{eq:cof} (recall that the
cofactor matrix of a symmetric matrix is symmetric). This suggests to define,
for $e\in\R^{3\times3}_{\mathrm{sym}}$ and $c,b\in\R^3$, the polynomial
\begin{equation}\label{eq:P}
P(e,c,b):=2\det(\Id+2e)\,Q_e(c\odot b)+\bigl(c^T\cof(\Id+2e)\,c\bigr)\,L_e(b\otimes b).
\end{equation}

\begin{lemma}[Reduction]\label{lemma:red}
For all $G\in\Rnn$ and $a,b\in\R^3$,
\begin{equation}\label{eq:red}
\det(G^TG)\,D^2W_0(G)[a\otimes b,a\otimes b]=P\bigl(\tfrac12(G^TG-\Id),\,G^Ta,\,b\bigr).
\end{equation}
Moreover $P(e,\lambda c,\mu b)=\lambda^2\mu^2P(e,c,b)$ and $P(0,c,b)=2Q(c,b)$.
\end{lemma}
\begin{proof}
By \eqref{eq:cof}, $\cof(G^TG)=(\cof G)^T\cof G$ and $(\cof G)\,G^T=(\det G)\Id$, hence
\[
(G^Ta)^T\cof(G^TG)\,G^Ta=a^T\bigl((\cof G)\,G^T\bigr)^T\bigl((\cof G)\,G^T\bigr)a
=\det(G^TG)|a|^2
\]
for every $G$. Inserting this, $S=c\odot b$ and
$K=\frac12|a|^2b\otimes b$ into \eqref{eq:D2W0} gives \eqref{eq:red}. The homogeneity is clear
since $Q_e$ is quadratic and $L_e$ is linear, and $P(0,c,b)=2Q_0(c\odot b)+|c|^2L_0(b\otimes b)$
is \eqref{eq:D2W0} at $G=\Id$, $a=c$.
\end{proof}

Frame indifference has been absorbed into the variables $(e,c)$: the rotational
degeneracy of $D^2W_0$ along $\SO$ is no longer present in $P$, and the only
remaining degeneracy of the zero set of $P(0,\cdot,\cdot)$ is the scaling in Lemma~\ref{lemma:red}.
We remove it by normalizing one component of $c$ and of $b$.

We also use a discrete symmetry. Let $\Pi$ be the permutation matrix exchanging $e_1$ and
$e_2$. By inspection of \eqref{eq:g}, $g(\Pi e\Pi)=g(e)$, hence $W_0(F\Pi)=W_0(F)$, and
$L_{\Pi e\Pi}(\Pi X\Pi)=L_e(X)$, $Q_{\Pi e\Pi}(\Pi X\Pi)=Q_e(X)$. Since moreover
$\det(\Id+2\Pi e\Pi)=\det(\Id+2e)$, $\cof(\Id+2\Pi e\Pi)=\Pi\cof(\Id+2e)\Pi$ and
$(\Pi c)\odot(\Pi b)=\Pi(c\odot b)\Pi$,
\begin{equation}\label{eq:sym}
P(\Pi e\Pi,\Pi c,\Pi b)=P(e,c,b).
\end{equation}
The symmetry exchanges the modes $R_1$ and $R_2$ and fixes $R_3$; so it suffices to
analyse $R_1$ and $R_3$.
For $y=(y_1,\dots,y_{10})\in\R^{10}$ let $E(y)$ be the symmetric matrix with
$(E_{11},E_{12},E_{13},E_{22},E_{23},E_{33})=(y_1,\dots,y_6)$, and
\begin{equation}\label{eq:charts}
\begin{aligned}
&c^{(1)}(y)=(1,y_7,y_8), && b^{(1)}(y)=(1,y_9,y_{10}),\\
&c^{(3)}(y)=(y_7,y_8,1), && b^{(3)}(y)=(1,1+y_9,y_{10}),
\end{aligned}
\end{equation}
and $P_k(y):=P\bigl(E(y),c^{(k)}(y),b^{(k)}(y)\bigr)$ for $k\in\{1,3\}$.
At $y=0$, $c^{(k)}\otimes b^{(k)}=R_k$.

\begin{lemma}[Nondegeneracy at the modes]\label{lemma:hess}
For $k\in\{1,3\}$ there are a polynomial $r_k(t,y)$ and a symmetric rational matrix
$A_k\in\R^{10\times10}$ such that
\[
P_k(ty)=t^2\,y^TA_ky+t^3r_k(t,y)\qquad\text{for all }t\in\R,\ y\in\R^{10},
\]
and $A_k-\frac14\Id$ is positive definite.
\end{lemma}
\begin{proof}
The expansion is an identity between explicit polynomials, which can easily be checked with a computer algebra system (we used Mathematica). One sees that the terms of order
$t^0$ and $t^1$ in $P_k(ty)$ vanish, and computes the matrices $A_k$ explicitly. Positive definiteness of $A_k-\frac14\Id$ is
checked in integer arithmetic, by verifying that the leading principal minors of the integer matrix $4A_k-\Id$ are positive.

For ease of verification we give here the two matrices.
{\scriptsize\[A_1=\begin{pmatrix}
36 & 0 & 0 & 3 & 0 & 0 & 0 & 0 & 0 & 0\\
0 & 24 & 0 & 0 & 0 & 0 & 0 & 0 & 4 & 0\\
0 & 0 & 25 & \frac12 & 1 & 0 & 0 & 1 & 0 & 5\\
3 & 0 & \frac12 & 13 & \frac12 & 0 & 0 & \frac12 & 0 & \frac12\\
0 & 0 & 1 & \frac12 & 25 & 0 & 0 & 1 & 0 & -3\\
0 & 0 & 0 & 0 & 0 & 12 & 0 & 0 & 0 & 0\\
0 & 0 & 0 & 0 & 0 & 0 & 2 & 0 & 0 & 0\\
0 & 0 & 1 & \frac12 & 1 & 0 & 0 & 2 & 0 & 2\\
0 & 4 & 0 & 0 & 0 & 0 & 0 & 0 & 2 & 0\\
0 & 0 & 5 & \frac12 & -3 & 0 & 0 & 2 & 0 & 4
\end{pmatrix},\qquad
A_3=\begin{pmatrix}
13 & 0 & \frac12 & 2 & \frac12 & 2 & 0 & 1 & -1 & 0\\
0 & 24 & 0 & 0 & 0 & -8 & 0 & 0 & 0 & 0\\
\frac12 & 0 & 46 & \frac12 & 22 & 0 & -2 & -2 & 0 & 0\\
2 & 0 & \frac12 & 13 & \frac12 & 2 & 1 & 0 & 1 & 0\\
\frac12 & 0 & 22 & \frac12 & 46 & 0 & -2 & -2 & 0 & 0\\
2 & -8 & 0 & 2 & 0 & 12 & 0 & 0 & 0 & 0\\
0 & 0 & -2 & 1 & -2 & 0 & 2 & 0 & 0 & 0\\
1 & 0 & -2 & 0 & -2 & 0 & 0 & 2 & 0 & 0\\
-1 & 0 & 0 & 1 & 0 & 0 & 0 & 0 & 2 & 0\\
0 & 0 & 0 & 0 & 0 & 0 & 0 & 0 & 0 & 2
\end{pmatrix}.\]}
\end{proof}

\begin{remark}
The vanishing of the $t^0$ term and of the $t^1$ term in the directions $y_7,\dots,y_{10}$
is automatic, since $P(0,c,b)=2Q(c,b)\ge0$ vanishes at $(c,b)=(c^{(k)}(0),b^{(k)}(0))$.
The genuinely nontrivial conditions are the vanishing of the $t^1$ term in the
directions $y_1,\dots,y_6$, i.e.\ $D^3W_0(\Id)[R_k,R_k,X]=0$ for symmetric $X$, and the
positivity of $A_k$.

The first condition is what the cubic terms of $g$ other than
$e_{11}e_{22}(e_{13}+e_{23})$ are needed for. Write $C$ for the cubic part
of $g$, and note that $c^{(k)}(0)\odot b^{(k)}(0)=\operatorname{sym}R_k=M_k$, with $M_k$ as in
\eqref{eq:Mk}. Since $q(M_k)=0$ and $b^{(k)}_3(0)=0$, the part of
$P(e,c^{(k)}(0),b^{(k)}(0))$ which is linear in $e$ is
$D^2C(e)[M_k,M_k]+2|c^{(k)}(0)|^2\,q(e,b^{(k)}(0)\otimes b^{(k)}(0))$, with $q(\cdot,\cdot)$
the polar form of $q$; the second term comes from the geometric nonlinearity $\frac12|a|^2b\otimes b$
in \eqref{eq:D2W0}. For $k=1$ the two terms are
$\partial^2_{e_{11}}C=2e_{22}$ and $2q(e,e_1\otimes e_1)=-2e_{22}$; for $k=3$ only
$e_{13}+e_{23}$ changes along $M_3$, at unit speed, so the first term is $2(e_{11}+e_{22}-4e_{12})$,
while $2q(e,(e_1+e_2)\otimes(e_1+e_2))=2(4e_{12}-e_{11}-e_{22})$.

The second condition is where the quartic term enters. The block of $A_k$ corresponding to the
directions $y_7,\dots,y_{10}$ is the matrix of the quadratic part of $2Q(c^{(k)}(y),b^{(k)}(y))$,
which is positive definite (Lemma~\ref{lemma:Q} proved the weaker condition that it is positive semidefinite). 
The block corresponding to $y_1,\dots,y_6$, instead, 
contains the positive definite contribution of the quartic term. Hence $A_k$ is positive definite 
if the coefficient of the quartic term is large enough. The fact that 3 suffices is checked by the positive-definiteness of the matrices above.
\end{remark}

\begin{remark}[Relation with the geometrically linear example]\label{rem:fromlinear}
The quadratic form $q$ is the one of the geometrically linear example of
Section~\ref{sec:linear}, and
\[
g(e)=g_{\mathrm{lin}}(e)+2e_{33}+e_{11}e_{22}(e_{11}+e_{22})+(e_{13}+e_{23})^2(e_{11}+e_{22}-4e_{12})+\tfrac94|e|^4,
\]
with $g_{\mathrm{lin}}$ the density of Theorem~\ref{thm:linear}. Thus $W_0$ is the linear example
evaluated at the strain $e=\frac12(F^TF-\Id)$, plus three modifications:
\emph{(i)} the prestress $2e_{33}=|Fe_3|^2-1$, which adds $|a|^2b_3^2$ to $Q$ and thereby removes the
zero $(e_1+e_2)\otimes e_3=R_3^T$ of $q(a\odot b)$. Some prestress is necessary: if $DW_0(\Id)=0$,
then $Q(a,b)$ depends only on $a\odot b$, so that $R_3$ and $R_3^T$ are zeros at the same time;
the first-order conditions at both then force $(e_1+e_2)\otimes(e_1+e_2)-2e_3\otimes e_3$ into the
kernel of $D^2g(0)$, which is not the case if the quadratic part of $g$ is $q$.
\emph{(ii)} The cubic correction compensates the geometric nonlinearity at the modes, as explained in
the previous remark. \emph{(iii)} The quartic term is enlarged, so that $A_k$ is positive definite.
\end{remark}

\section{Rank-one convexity near \texorpdfstring{$\SO$}{SO(3)}}\label{sec:rc}

\begin{lemma}[Nondegenerate minima]\label{lemma:MB}
Let $\Theta$,  $\alpha:\R^N\to\R$,
$\beta\in C^0(\R\times \R^N)$, $m>0$
be such that for all $t\in\R$, $y\in\R^N$ one has
$\Theta(ty)=t^2\alpha(y)+t^3\beta(t,y)$,
and $\alpha(y)\ge m|y|^2$.
% and some $\beta\in C^0(\R\times \R^N)$.
Then $\Theta\ge0$ in a
neighbourhood of $0$.
\end{lemma}
\begin{proof}
Let $C:=\max\{|\beta(t,u)|:|t|\le1,\ |u|=1\}$. For $y=\varrho u$ with $|u|=1$ and
$0<\varrho\le\min\{1,m/(C+1)\}$ we get
$\Theta(y)=\varrho^2\bigl(\alpha(u)+\varrho\,\beta(\varrho,u)\bigr)\ge\varrho^2(m-\varrho C)\ge0$,
and $\Theta(0)=0$.
\end{proof}

\begin{proposition}[Local rank-one convexity]\label{prop:rc}
There is $\delta>0$ such that
$D^2W_0(F)[a\otimes b,a\otimes b]\ge0$ for all $F\in (\SO)_\delta$ and all $a,b\in\R^3$.
In particular, $t\mapsto W_0(F+ta\otimes b)$ is convex on every interval on
which $F+ta\otimes b\in (\SO)_\delta$.
\end{proposition}
\begin{proof}
\emph{Step 1: There is a neighbourhood $V$ of $0$ in $\R^{3\times3}_{\mathrm{sym}}$
such that $P(e,c,b)\ge0$ for all $e\in V$ and $c,b\in\R^3$.} By the homogeneity of
Lemma~\ref{lemma:red} it suffices to consider $(c,b)\in\Sigma:=S^2\times S^2$. Since $\Sigma$
is compact, it suffices to show that every $(c_0,b_0)\in\Sigma$ has a neighbourhood $O$ in
$\R^{3\times3}_{\mathrm{sym}}\times\Sigma$ of $(0,c_0,b_0)$ with $P\ge0$ on $O$.

If $P(0,c_0,b_0)>0$ this follows from continuity. Otherwise $Q(c_0,b_0)=0$ and by
Lemma~\ref{lemma:Q} $c_0\otimes b_0$ is a multiple of some $R_k$. Consider $k=1$, so
$c_0=\pm e_1$, $b_0=\pm e_1$. Near $(0,c_0,b_0)$ we have $c_1b_1\ne0$, and by homogeneity
\[
P(e,c,b)=c_1^2b_1^2\,P_1(y),\qquad y:=\bigl(e_{11},e_{12},e_{13},e_{22},e_{23},e_{33},\tfrac{c_2}{c_1},\tfrac{c_3}{c_1},\tfrac{b_2}{b_1},\tfrac{b_3}{b_1}\bigr),
\]
where $y$ depends continuously on $(e,c,b)$ and vanishes at $(0,c_0,b_0)$. By
Lemmas~\ref{lemma:hess} and~\ref{lemma:MB} (with $m=\frac14$), $P_1\ge0$ near $0$, and the
claim follows. The case $k=2$ reduces to $k=1$ by \eqref{eq:sym}, since
$(\Pi c_0,\Pi b_0)=(\pm e_1,\pm e_1)$ and $(e,c,b)\mapsto(\Pi e\Pi,\Pi c,\Pi b)$ is a linear
isomorphism. For $k=3$, $c_0=\pm e_3$, $b_0=\pm(e_1+e_2)/\sqrt2$; then
$c_3b_1\ne0$ near $(c_0,b_0)$ and $P(e,c,b)=c_3^2b_1^2P_3(y)$ with
$y=(e_{11},\dots,e_{33},\frac{c_1}{c_3},\frac{c_2}{c_3},\frac{b_2}{b_1}-1,\frac{b_3}{b_1})$.

\emph{Step 2.} The set $U:=\{F\in\Rnn:\frac12(F^TF-\Id)\in V,\ \det F\ne0\}$ is open and
contains the compact set $\SO$, hence $(\SO)_\delta\subset U$ for some $\delta>0$. For
$F\in U$, Lemma~\ref{lemma:red} and Step~1 give
$D^2W_0(F)[a\otimes b,a\otimes b]=P(\frac12(F^TF-\Id),F^Ta,b)/\det(F^TF)\ge0$. The last
assertion follows since $\frac{d^2}{dt^2}W_0(F+ta\otimes b)=D^2W_0(F+ta\otimes b)[a\otimes b,a\otimes b]$.
\end{proof}

\section{Failure of quasiconvexity at the identity}\label{sec:qc}

\begin{proposition}\label{prop:notqc}
Let $\varphi(x):=(\sin x_1,\ \sin x_2,\ -\sin(x_1+x_2))$, $x\in\R^3$. Then
$|\nabla\varphi|\le2$ and
\[
\avg_{(0,2\pi)^3}W_0(\Id+\varepsilon\nabla\varphi)\,dx=-\tfrac14\varepsilon^3+O(\varepsilon^4)
\qquad(\varepsilon\to0).
\]
In particular, for every $\rho>0$ there is $\varepsilon>0$ with
$\varepsilon\nabla\varphi\in B_\rho(0)$ pointwise and
$\avg W_0(\Id+\varepsilon\nabla\varphi)\,dx<0=W_0(\Id)$.
\end{proposition}
\begin{proof}
The computation is similar to the one in Lemma~\ref{lemma:linnotqc}.
In the notation of Lemma~\ref{lemma:L},
$\nabla\varphi(x)\in L$ with $r=\cos x_1$, $s=\cos x_2$, $u=-\cos(x_1+x_2)$, so
$|\nabla\varphi|^2=r^2+s^2+2u^2\le4$. By Lemma~\ref{lemma:L},
\[
\avg_{(0,2\pi)^3} W_0(\Id+\varepsilon\nabla\varphi)\,dx=-2\varepsilon^2\avg_{(0,2\pi)^3} rs\, dx
+\varepsilon^3\avg_{(0,2\pi)^3} rsu\, dx+O(\varepsilon^4),
\]
and $\avg\cos x_1\cos x_2\, dx=0$,
$\avg\cos x_1\cos x_2\cos(x_1+x_2)\, dx=\frac14$, so $\avg rsu\, dx=-\frac14$.
\end{proof}

By Lemma~\ref{lemma:periodic}, any continuous function which coincides with
$W_0$ on $B_\rho(\Id)$ for some $\rho>0$ is not quasiconvex (at $\Id$).

\begin{remark}
The second variation $\frac12D^2W_0(\Id)$ restricted to $L$ is $-2rs$, which is
indefinite on $L$ but vanishes on the three rank-one lines of $L$; the cubic
term $rsu$ couples the three modes since $(1,0)+(0,1)=(1,1)$ for their
wave vectors. This is the mechanism of \cite{Sverak92}. By Kristensen
\cite{Kristensen99} a $C^2$ function satisfying the \emph{strict}
Legendre--Hadamard condition at every point is locally quasiconvex; here the
Legendre--Hadamard condition at $\Id$ is degenerate exactly along the modes.
\end{remark}

\section{Globalization}\label{sec:global}

With $T$ as in \eqref{eq:Taffine}, for $\sigma>0$, $F\in\Rnn$ and
$(X,Y,d)\in\R^9\times\R^9\times\R$ let
\begin{equation}\label{eq:gamma}
\begin{gathered}
g_0(F):=|F-\cof F|^2+(\det F-1)^2,\qquad\gamma_\sigma(F):=g_0(F)+\sigma|F|^2,\\
\Gamma_\sigma(X,Y,d):=|X-Y|^2+(d-1)^2+\sigma|X|^2,\qquad\text{so that }\gamma_\sigma=\Gamma_\sigma\circ T.
\end{gathered}
\end{equation}
$\Gamma_\sigma$ is a convex quadratic polynomial on $\R^{19}$.

\begin{lemma}\label{lemma:glob}
Let $g_0$ and $\gamma_\sigma$ be as in \eqref{eq:gamma}.
\begin{enumerate}
\item $g_0\ge0$, and $g_0(F)=0$ if and only if $F\in\SO$.
\item For every $\sigma>0$, $\gamma_\sigma(QF)=\gamma_\sigma(F)$ for all $Q\in\SO$, $F\in\Rnn$;
$\gamma_\sigma=3\sigma$ on $\SO$; $\gamma_\sigma\ge g_0$ and $\gamma_\sigma(F)\ge\sigma|F|^2$.
\item For every $\sigma>0$, $F\in\Rnn$ and $a,b\in\R^3$, the function
$t\mapsto\gamma_\sigma(F+ta\otimes b)$ is a polynomial of degree at most two, with second
derivative at least $2\sigma|a\otimes b|^2$.
\item For every open set $U\supset\SO$ there is $\sigma>0$ such that
$\{\gamma_\sigma\le5\sigma\}\subset U$.
\item There is $C\ge0$ such that $D^2W_0(G)[A,A]\ge-C(1+|G|^2)^3|A|^2$ for all $G,A\in\Rnn$.
\end{enumerate}
\end{lemma}
\begin{proof}
(i) Clearly $g_0\ge0$. If $g_0(F)=0$, then $\cof F=F$ and $\det F=1$, so
$F^TF=F^T\cof F=\Id$ by \eqref{eq:cof}, and $F\in\SO$. Conversely $\cof Q=Q$ and $\det Q=1$
for $Q\in\SO$.

(ii) follows from \eqref{eq:cofQ}, $|QX|=|X|$, (i) and $|Q|^2=3$ for $Q\in\SO$.

(iii) By \eqref{eq:Taffine}, $\gamma_\sigma(F+ta\otimes b)=\Gamma_\sigma\bigl(T(F)+t\,(a\otimes b,V_{\cof},V_{\det})\bigr)$,
whose second derivative in $t$ is
$2|a\otimes b-V_{\cof}|^2+2V_{\det}^2+2\sigma|a\otimes b|^2$.

(iv) The set $Z:=\{F:|F|^2\le5\}\setminus U$ is compact and $g_0>0$ on $Z$ by (i); let
$m:=\min_Zg_0>0$ (or $m:=1$ if $Z=\emptyset$) and $\sigma:=m/6$. If $\gamma_\sigma(F)\le5\sigma$,
then $\sigma|F|^2\le5\sigma$ and $g_0(F)\le5\sigma<m$ by (ii), hence $|F|^2\le 5$ and $F\notin Z$,
i.e.\ $F\in U$.

(v) Since $W_0$ is a polynomial of degree $8$, $D^2W_0(G)[A,A]$ is a polynomial in $(G,A)$
which is homogeneous of degree two in $A$ and of degree at most six in $G$.
\end{proof}

\begin{remark}
The determinant term in $g_0$ is essential: $|F-\cof F|^2$ vanishes also at $F=0$, where $W_0$
need not be rank-one convex.
\end{remark}

\subsection*{Proof of Theorem~\ref{thm:main}}
Let $\delta>0$ be as in Proposition~\ref{prop:rc}, $\sigma$ as in Lemma~\ref{lemma:glob}(iv)
with $U=(\SO)_\delta$, $C$ as in Lemma~\ref{lemma:glob}(v), and
\[
\chi(s):=\int_{4\sigma}^s4\eta(r)(r-4\sigma)^3\,dr,\qquad K:=\frac{27C}{\sigma^4}+1,\qquad
\Phi:=\chi\circ\gamma_\sigma,\qquad W:=W_0+K\Phi ,
\]
where $\eta\in C^\infty(\R)$ is nondecreasing, $\eta=0$ on $(-\infty,4\sigma]$ and $\eta=1$ on
$[5\sigma,\infty)$. Then $\chi\in C^\infty(\R)$, $\chi=0$ on $(-\infty,4\sigma]$,
$\chi'(s)=4(s-4\sigma)^3$ for $s\ge5\sigma$, and $\chi$ is nondecreasing and convex, since
$\chi''(s)=4\eta'(s)(s-4\sigma)^3+12\eta(s)(s-4\sigma)^2\ge0$ for $s\ge4\sigma$. Hence
$\Phi=(\chi\circ\Gamma_\sigma)\circ T$ is polyconvex, of class $C^\infty$, and
frame-indifferent by Lemma~\ref{lemma:glob}(ii), and
$W\in C^\infty(\Rnn)$ satisfies \eqref{eq:fi}. Moreover $W=W_0$ on the open set
$\{\gamma_\sigma<4\sigma\}$, which contains $\SO$ since $\gamma_\sigma=3\sigma$ there.

\emph{Rank-one convexity.} Fix $F$ and $A=a\otimes b$, let $f(t):=W(F+tA)$,
$h(t):=\gamma_\sigma(F+tA)$, and $G:=F+tA$. Then
\[
f''(t)=D^2W_0(G)[A,A]+K\bigl(\chi''(h)h'^2+\chi'(h)h''\bigr),
\]
and $\chi',\chi''\ge0$, $h''\ge2\sigma|A|^2$ by Lemma~\ref{lemma:glob}(iii).
If $h(t)<5\sigma$, then $G\in (\SO)_\delta$ by the choice of $\sigma$, and $f''(t)\ge0$ by
Proposition~\ref{prop:rc}. If $h(t)\ge5\sigma$, then $h-4\sigma\ge h/5$, so
$\chi'(h)h''\ge\frac{4}{125}h^3\cdot2\sigma|A|^2$; and $\sigma|G|^2\le h$ gives
$1+|G|^2\le\frac{6h}{5\sigma}$. By Lemma~\ref{lemma:glob}(v),
\[
f''(t)\ge-C\Bigl(\frac{6h}{5\sigma}\Bigr)^3|A|^2+\frac{8K\sigma}{125}h^3|A|^2
=\Bigl(\frac{8K\sigma}{125}-\frac{216C}{125\sigma^3}\Bigr)h^3|A|^2\ge0
\]
by the choice of $K$. Hence $f''\ge0$ on $\R$ and $f$ is convex.

\emph{Failure of quasiconvexity.} Choose $\rho>0$ with $B_\rho(\Id)\subset\{\gamma_\sigma<4\sigma\}$.
By Proposition~\ref{prop:notqc} there is $\varepsilon>0$ with
$\Id+\varepsilon\nabla\varphi(x)\in B_\rho(\Id)$ for all $x$ and
$\avg W(\Id+\varepsilon\nabla\varphi)\,dx=\avg W_0(\Id+\varepsilon\nabla\varphi)\,dx<0=W(\Id)$.
By Lemma~\ref{lemma:periodic}, $W$ is not quasiconvex. \qed

\begin{remark} %[Further properties]
Since only the values of $W$ near $\SO$ matter for the failure of
quasiconvexity, one may add any frame-indifferent polyconvex function vanishing
near $\SO$. For instance, adding $\theta(\det F)$ with $\theta$ convex, $\theta=0$ on
$[\frac12,\infty)$ and $\theta(t)\to\infty$ as $t\downarrow0$ (and $\theta=\infty$ on
$(-\infty,0]$) gives an energy with $W(F)\to\infty$ as $\det F\downarrow0$.
\end{remark}

\begin{remark}
The function $W$ coincides with $W_0$ near $\SO$,
but it is not a polynomial. A polynomial example is obtained by replacing the cutoff
$\chi\circ\gamma_\sigma$ by a power of a polynomial which vanishes exactly on $\SO$.
For example, one can take
$W_1:=W_0+K\zeta^4$
where
\begin{equation*}%\label{eq:zeta}
\begin{gathered}
\zeta(F):=4|F-\cof F|^2+(2\det F-3)^2+2|F|^2-7
=
4g_0(F)+2\bigl(|F|^2-2\det F-1\bigr)
\end{gathered}
\end{equation*}
and $K$ is chosen sufficiently large.
\end{remark}

\section*{Acknowledgements}
This work was partially supported
by the Deutsche Forschungsgemeinschaft (DFG) through project CRC~1720 -- 539309657, project
SPP~2256 -- 441468770, and project SPP~2256 -- 523956128.
The authors acknowledge Anthropic’s Claude for Scientists program, which provided reduced-cost access to Claude Code used in this research.

\section*{Statement on the use of artificial intelligence}
During the preparation of this work the authors used the large language model Claude
Opus 5.5 (Anthropic), accessed through the agentic coding tool Claude Code. The model was used
(i) to write the Lean~4 formalization of Theorems~\ref{thm:main} and~\ref{thm:linear}
\cite{LeanCode}, which was produced entirely by the model, with guidance from the authors on
the statements and the proof strategy; and (ii) to assist in the development of the
construction and of the proofs, for symbolic and numerical computations, and in drafting
and revising the text of the manuscript. The authors have checked all mathematical arguments
in the paper and the formal statements in the Lean code (file \texttt{Statement.lean}), which
have been verified by the Lean proof checker. The authors reviewed and edited all content
generated by the model and take full responsibility for the content of this article. The
model is not an author of this work.

\printbibliography

\end{document}